\documentclass[a4paper, 11pt]{article}
\usepackage{amsmath, amssymb,amscd, amsthm}
\usepackage[mathscr]{eucal}
\usepackage{graphics}
\usepackage{fullpage}
\usepackage{url}
\newcommand\cyr{%
 \renewcommand\rmdefault{wncyr}%
 \renewcommand\sfdefault{wncyss}%
 \renewcommand\encodingdefault{OT2}%
\normalfont\selectfont} \DeclareTextFontCommand{\textcyr}{\cyr}

\newtheorem{theorem}{Theorem}
\newtheorem{lemma}[theorem]{Lemma}

\newtheorem{proposition}[theorem]{Proposition}

\newtheorem{problem}[theorem]{Problem}

\def\O{\mathcal O
}

\begin{document}

\title{Orthonormal Systems in Linear Spans}

\author{Allison Lewko\thanks{Supported by a Microsoft Research PhD Fellowship.} \and Mark Lewko}

\date{}
\maketitle
\begin{abstract}We show that any $N$-dimensional linear subspace of $L^2(\mathbb{T})$ admits an orthonormal system such that the $L^2$ norm of the square variation operator $V^2$ is as small as possible. When applied to the span of the trigonometric system, we obtain an orthonormal system of trigonometric polynomials with a $V^2$ operator that is considerably smaller than the associated operator for the trigonometric system itself.
\end{abstract}

\section{Introduction}

Let $\mathbb{T}$ denote a probability space and $\Phi := \{\phi_n(x)\}_{n=1}^{N}$ an orthonormal system (ONS) of functions from $\mathbb{T}$ to $\mathbb{R}$. One is often interested, usually motivated by questions regarding almost everywhere convergence, in the behavior of the maximal function

\[\mathcal{M}f := \max_{\ell \leq N}\left| \sum_{n=1}^{\ell} a_n \phi_n  \right|. \]

For an arbitrary ONS, the Rademacher-Menshov theorem states that $||\mathcal{M}f ||_{L^2} \ll \log(N) ||f||_{L^2}$, where the $\log(N)$ factor is known to be sharp. One however can do much better for many classical systems, for instance one can replace $\log(N)$ with an absolute constant in the case of the trigonometric system (the Carleson-Hunt inequality).  More recently, there has been interest in variational refinements of these maximal results. Define the $r$-th variation operator
\[  \mathcal{V}^{r}f  := \left(\max_{\pi \in \mathcal{P}_{N}} \sum_{I \in \pi } \left| \sum_{n\in I} a_n\phi_n \right|^r \right)^{1/r} \]
where $\mathcal{P}_N$ denotes the set of partitions of $[N]$ into subintervals.  Clearly, $|\mathcal{M}f | \leq |\mathcal{V}^{r}f |$ for all $r< \infty$. In the case of trigonometric system it has been shown that $||\mathcal{V}^{r}f||_{2} \ll ||f||_{2} $ for $r>2$ (see \cite{OSTTW}), and  $||\mathcal{V}^{2}f||_{2} \ll \sqrt{\log(N)}||f||_{2} $ (see \cite{LewkoJFA}), where the factor of $\sqrt{\log(N)}$ is optimal. This later inequality has some applications to sieve theory \cite{LewkoJNT}. The factor of $\sqrt{\log(n)}$ is rather unfortunate, leading to inefficiencies in these applications. It is likely that this factor can be improved for the functions arising in the applications, for instance, if the Fourier support of $f$ is contained in certain arithmetic sets. This is a potential route towards improving the estimates in \cite{LewkoJNT}. Some results in this direction can be found in section 7 of \cite{LewkoJFA}.

In a different direction, it seems that the $\sqrt{\log(n)}$ factor might also be an eccentricity of the standard ordering of the trigonometric system. In \cite{LewkoJFA} the following problem was posed:

\begin{problem}\label{prob:main}Is there a permutation $\sigma : [N] \rightarrow [N]$ such that the reordering of the trigonometric system $\Phi:= \{ \phi_{n} = e(\sigma(n)x)\}$ (where $e(x):=e^{2 \pi i x})$ satisfies
\[ ||\mathcal{V}^{2}f||_{2} \ll o(\sqrt{\log(N)}) ||f||_{2} \]
for all $f$ in the span of the system?
\end{problem}

This problem can be thought of as a variational variant of Garsia's conjecture. We refer the reader to \cite{Bour} and \cite{LewkoJFA} for discussion of these and related problems. In support of an affirmative answer, it was proved in \cite{LewkoJFA} that given a function $f=\sum_{n=1}^{N} a_n e(nx)$, there exists a permutation $\sigma : [N] \rightarrow [N]$ such that reordered trigonometric system satisfies $||\mathcal{V}^{2}f||_{2} \ll  \sqrt{\log \log(N)} ||f||_{2} $. There the permutation is allowed to depend on the function, while the above problem seeks a permutation that works for all functions simultaneously.

In this paper, we will study the following related problem. Given an ONS $\Phi := \{\phi_n(x)\}_{n=1}^{N}$ and a $N \times N$ orthogonal matrix $O = \{o_{i,n}\}_{1\leq i,n \leq N} $, we define a new ONS, $\Psi := \{\psi_n(x)\}_{n=1}^{N}$, by
\[ \psi_n(x) := \sum_{i=1}^{N} o_{i,n} \phi_i(x). \]
This new system will span the same space as the original system. Conversely, every such ONS can be obtained from some element of the orthogonal group, $\mathcal{O}(N)$. Let us write $\Phi(O):= \Psi$. Furthermore, in what follows $Q$ will denote a measurable subset of $\mathcal{O}(N)$ and $\mathbb{P}[Q]$ will denote the Haar measure of $Q$.

\begin{theorem}\label{thm:main}Given an ONS $\Phi := \{\phi_n(x)\}_{n=1}^{N}$ from $\mathbb{T}$ to $\mathbb{R}$, there exists an alternate ONS $\Phi(O)$ that spans the same space, and satisfies
\begin{equation}
\label{eq:mainvar}||\mathcal{V}^{2}f||_{2} \ll \sqrt{\log \log(N)} ||f||_{2}
\end{equation}
for all $f$ in the span. In fact, the conclusion holds for all $O \in Q$ for some $Q \subset \mathcal{O}(N)$ with $\mathbb{P}[Q] \geq 1 - Ce^{-cN^{2/5}}$ (for some absolute positive constants $C,c$).
\end{theorem}

If we take $\Phi:=\{e(nx) \}_{n=1}^{N}$, then this produces an ONS of trigonometric polynomials (spanning the same space as the trigonometric system) with much smaller square variation than the trigonometric system. Strictly speaking, Theorem \ref{thm:main} is stated for real valued ONS, but the result for the trigonometric system can be obtained by splitting into real and imaginary parts and noting the corresponding result holds on each with large probability.  We note that Problem \ref{prob:main} asks for a similar conclusion where $O$ is restricted to be a permutation matrix instead of just an orthogonal matrix.

Theorem \ref{thm:main} is sharp. Consider an ONS of independent, mean zero, variance one Gaussians, $\{g_{i}\}_{i=1}^N$. Notice that applying an orthogonal transformation to this system leaves it metrically unchanged. On the other hand, we have that $\max_{\pi \in \mathcal{P}_{N}} \sum_{I \in \pi } \left| \sum_{n\in I} g_n \right|^2 \sim 2 N \log\log(N) $ (almost surely) from the variational law of the iterated logarithm \cite{LewkoProb}.

Let us briefly outline the key idea in the proof of Theorem \ref{thm:main}. In \cite{LewkoJFA}, we proved an estimate of the form (\ref{eq:mainvar}) for systems of bounded independent random variables (see Theorem 9).  The key ingredient in that case is that for every $f$ in the span of the system we have the sub-gaussian tail estimate $||f||_{\mathcal{G}} \ll ||f||_{2}$ (where $||\cdot ||_{\mathcal{G}}$ is the Orlicz space norm associated to $e^{x^2}-1$). This clearly cannot hold in the setting of Theorem \ref{thm:main}, since any $L^2$ function can be in the span of the system. However, we will show that a function $f$ in the span of a generic basis $\Phi(O)$ can be split $f=G+E$, where $G$ satisfies a sub-Gaussian tail inequality and $E$ has small $L^2$ norm (decreasing with the size of the Fourier support of $f$). More precisely, we will prove (note that we abuse the notation $c$ below to denote multiple distinct constants):

\begin{proposition}\label{prop:decomp0} For $N$  fixed, let $\Phi=\{\phi_n(x)\}_{n=1}^N$ be an ONS such that $\sum_{n=1}^{N}|\phi_n(x)|^2 \leq N$ holds (pointwise). There exists $Q \subset \mathcal{O}(N)$ with $\mathbb{P}[Q] \geq 1 - Ce^{-cN^{2/5}}$ such that for $O\in Q$, we have that the associated ONS $\Phi(O)=\{ \psi_{n}\}_{n=1}^{N}$ satisfies the following property. For any $f=\sum a_n \psi_n$, letting $m$ denote support$(\{a_n\})$ (the number of nonzero $a_i$ values), we have that the function defined by
\[ f :=  \sum a_n \psi_n(x)   \]
can be decomposed as $f := G + E$ where $|| G ||_{\mathcal{G}} \ll ||f||_{2}$ and $|| E ||_{2} \ll \left(\frac{m}{N}\right)^{c}||f||_{2} $ for some universal constant $c>0$.
\end{proposition}

See Proposition \ref{prop:decomp} below, which gives a stronger maximal form of this statement. The condition $\sum_{n=1}^{N}|\phi_n(x)|^2 \leq N$ can usually be removed in applications (such as Theorem 1) by a change of measure argument (see Lemma \ref{chgMeas}).  It seems likely that this decomposition may have other applications.

\section{Preliminaries}
We need to define several different norms on the space of functions from $\mathbb{T}$ to $\mathbb{R}$. First, for a positive constant $c$, let $|| \cdot ||_{\mathcal{G}(c)}$ denote the norm of the Orlicz space associated to the convex function $e^{cx^2}-1$. That is, \[||f||_{\mathcal{G}(c)} := \inf_{\lambda \in \mathbb{R}^{+}} \left\{ \int e^{c|f/\lambda|^2}-1 \leq 1 \right\}.\] When we write $||\cdot ||_{\mathcal{G}}$ with the specification of $c$ omitted, we mean $c = 1$.

We next define the convex function
\[\Gamma_{K}(t):= \left\{
                   \begin{array}{ll}
                     e^{t^2}-1, & \hbox{$|t|\leq K$} \\
                     e^{K^2}t^2 + e^{K^2}(1-K^2) - 1, & \hbox{$|t| \geq K$}
                   \end{array}
                 \right.
\]
and denote the associated Orlitz norm $||\cdot ||_{\Gamma_K}$.  We then have

\begin{lemma}\label{lem:Gamma1} When $K \geq 1$, for all $t$ we have that
\[\Gamma_K(t) \leq e^{t^2} -1\]
\[\Gamma_K(t) \leq e^{K^2}t^2.\]
It follows that for $f: \mathbb{T}\rightarrow \mathbb{R}$ we have $||f||_{\Gamma_K} \leq  ||f||_{\mathcal{G}}$ and $||f||_{\Gamma_K} \leq e^{K^2/2}||f||_{L^2}$.
\end{lemma}

\begin{proof} We first prove $\Gamma_K(t) \leq e^{t^2}-1$ for all $t$. For $t$ such that $|t| \leq K$, this is clear since $\Gamma_K(t) = e^{t^2}-1$. We consider $t$ such that $|t| \geq K$. Then $\Gamma_K(t)  = e^{K^2}t^2 + e^{K^2} (1-K^2)-1$, so we must show that $e^{K^2}t^2 + e^{K^2} (1-K^2) \leq e^{t^2}$. We note that for all real $x \geq 0$, $1+x \leq e^x$. Applying this to the quantity $t^2 -K^2 +1 >0$, we have:
\[ e^{K^2}t^2 + e^{K^2} (1-K^2) = e^{K^2} (t^2 - K^2 +1) \leq e^{K^2} e^{t^2 - K^2} = e^{t^2},\]
as required.

We let $f$ be a function from $\mathbb{T}$ to $\mathbb{R}$.
For any fixed positive real number $\lambda$ such that $\int e^{|f/\lambda|^2} -1 \leq 1$ (i.e. $\lambda \geq ||f||_{\mathcal{G}}$), we have
\[ \int \Gamma_K(f/\lambda)\leq \int e^{|f/\lambda|^2}-1 \leq 1,\]
since $\Gamma_K(t) \leq e^{t^2}-1$ for all $t$. This shows that $\lambda \geq ||f||_{\Gamma_K}$, hence $||f||_{\Gamma_K} \leq  ||f||_{\mathcal{G}}$.

We next prove $\Gamma_K(t) \leq e^{K^2} t^2$. We first consider $t$ such that $|t| \geq K$. In this case, $\Gamma_K(t) = e^{K^2} t^2 + e^{K^2}(1-K^2) -1$. Since $K\geq 1$, we see that $e^{K^2} (1-K^2) < 0$, so $\Gamma_K(t) \leq e^{K^2} t^2$ follows. For $t$ such that $|t| \leq K$, we have $\Gamma_K(t) = e^{t^2}-1$, so we must show that $e^{t^2}-1 \leq e^{K^2} t^2$ for $|t| \leq K$.

We consider $\frac{e^{t^2}-1}{t^2}$ as a function of $t$ for $t \geq 0$. Its derivative is:
\[ 2\left( t^{-1} e^{t^2} - t^{-3} e^{t^2} + t^{-3}\right).\]
We observe that this is always non-negative. To see this, consider multiplying the quantity by $t^3$ to obtain $2(t^2 e^{t^2}  - e^{t^2} + 1)$. Non-negativity then follows from the inequality $1+ xe^x \geq e^x$ for all real $x \geq 0$. (This inequality can be proved by noting that $xe^x \geq \int_{0}^x e^u du$.) Hence $\frac{e^{t^2}-1}{t^2}$ is a non-decreasing function of $t$ in the range $0 \leq t \leq K$, so it suffices to consider the value at $t=K$, which is $K^{-2} (e^{K^2}-1)$. Since $K\geq 1$, this is $< e^{K^2}$, as required.

For $f: \mathbb{T}\rightarrow \mathbb{R}$, we consider $\lambda := e^{K^2/2} ||f||_{L^2}$. Then
\[ \int \Gamma_K(f/\lambda) \leq \int e^{K^2} \frac{f^2}{\lambda^2}  = \frac{e^{K^2}}{\lambda^2} ||f||_{L^2}^2 =1,\]
since $\Gamma_K(t) \leq e^{K^2} t^2$. Thus, $||f||_{\Gamma_K} \leq e^{K^2/2}||f||_{L^2}$.
\end{proof}

\begin{lemma}\label{lem:Gamma2} For any (measurable) $f : \mathbb{T}\rightarrow \mathbb{R}$, we can decompose $f = f_1 + f_2$ such that
\[ ||f_1||_{\mathcal{G}} \ll ||f||_{\Gamma_K} \text{  and }\]
\[ ||f_2||_{L^2} \ll  e^{-cK^2} ||f||_{\Gamma_K},\]
for some universal constant $c > 0$.
\end{lemma}

\begin{proof} Given $f$, we define $\gamma := 2||f||_{\Gamma_K}$ to simplify our notation. We then set:
\[f_1 := f \cdot \mathbb{I}_{|\frac{f}{\gamma}|\leq K} \text{  and  } f_2 := f \cdot \mathbb{I}_{|\frac{f}{\gamma}|\geq K},\]
where $\mathbb{I}_{S}$ for a set $S\subset \mathbb{T}$ denotes the indicator function for that set.
By definition of $\gamma = 2||f||_{\Gamma_K} > ||f||_{\Gamma_K}$, we have that
\begin{equation}\label{eq:pieces}
\int \Gamma_K(f/\gamma) = \int \left(e^{|f/\gamma|^2}-1\right)\cdot  \mathbb{I}_{|\frac{f}{\gamma}|\leq K} + \int \left(e^{K^2}f^2 / \gamma^2 + e^{K^2}(1-K^2) -1 \right) \cdot \mathbb{I}_{|\frac{f}{\gamma}|\geq K} \leq 1.
\end{equation}
Since this is a sum of two non-negative quantities, this implies
\[ \int \left(e^{|f/\gamma|^2}-1\right)\cdot  \mathbb{I}_{|\frac{f}{\gamma}|\leq K} \leq 1.\]
This is equivalent to:
\[ \int e^{|f_1/\gamma|^2}-1 \leq 1,\]
and so $||f_1||_{\mathcal{G}} \leq \gamma \ll ||f||_{\Gamma_K}$.

Again considering (\ref{eq:pieces}), we also have
\[ \int \left(e^{K^2}f^2 / \gamma^2 + e^{K^2}(1-K^2) -1 \right) \cdot \mathbb{I}_{|\frac{f}{\gamma}|\geq K} \leq 1.\]
We let $\mu\left( \left|\frac{f}{\gamma}\right|\geq K\right)$ denote the measure of the set in $\mathbb{T}$ on which $|\frac{f}{\gamma}|\geq K$. We can then rewrite the above as:
\begin{equation}\label{eq:pieces2}
\mu\left( \left|\frac{f}{\gamma}\right|\geq K\right) (e^{K^2}(1-K^2) -1 ) +  \int e^{K^2}f_2^2 / \gamma^2 \leq 1.
\end{equation}

Now, since $\int \Gamma_K (f/\gamma) \leq 1$ and $\Gamma_K(f/\gamma) \geq e^{K^2}-1$ whenever $|f/\gamma|\geq K$, we must have
\[\mu\left( \left|\frac{f}{\gamma}\right|\geq K\right) (e^{K^2}-1) \leq 1.\]
Thus, $\mu\left( \left|\frac{f}{\gamma}\right|\geq K\right) \leq \frac{1}{e^{K^2}-1}$.
Combining this with (\ref{eq:pieces2}), we have
\[ \int e^{K^2}f_2^2 / \gamma^2 \leq 1 + \mu\left( \left|\frac{f}{\gamma}\right|\geq K\right)(e^{K^2}(K^2-1)+1) \ll K^2,\]
and hence
\[ ||f_2||_{L^2}^2 \ll K^2 e^{-K^2} \gamma^2, \]
implying that $||f_2||_{L^2} \ll  e^{-cK^2} ||f||_{\Gamma_K}$ for some universal constant $c > 0$.

\end{proof}

%
%
%

Finally, we note the following.

\begin{lemma}\label{chgMeas}It suffices to prove Theorem 1 with the restriction that $\sum_{n=1}^{N}|\phi_n(x)|^2 \leq N$.
\end{lemma}

\begin{proof}Consider an arbitrary ONS $\Phi := \{\phi_n\}_{n=1}^{N}$ and define $\nu(x)= N^{-1}\sum_{n=1}^{N} |\phi_n(x)|^2 $. Fix $O \in \mathcal{O}(N)$. Define $\tilde{\Phi}:=\Phi(O)$. Furthermore, consider the ONS $\Psi$ defined on $\mathbb{T}$ (with the measure induced by integration against $\nu(x)$) by $\psi_n(x) := \nu^{-1/2}(x) \phi_n(x)$. Furthermore, define $\tilde{\Psi}=\Psi(O)$. We have the trivial identity

\[ \int  \max_{\pi \in \mathcal{P}_{N}} \sum_{I \in \pi } \left| \sum_{n\in I} a_n \tilde{\phi}_n(x) \right|^2  = \int  \max_{\pi \in \mathcal{P}_{N}} \sum_{I \in \pi } \left| \sum_{n\in I} a_n \tilde{\psi}_n(x)\right|^2 \nu(x).\]

Thus, the conclusion of Theorem 1 holds for $\Phi$ if and only if it holds for $\Psi$. However $\sum_{n=1}^{N} |\psi_n(x)|^2 \leq N$ by construction.
\end{proof}

\section{Probabilistic Methods}

In this section we establish the following result:

\begin{proposition}\label{prop:main} For $N$  fixed, let $\{\phi_n(x)\}_{n=1}^N$ be an ONS such that $\sum_{n=1}^{N}|\phi_n(x)|^2 \leq N$. Define for each $1 \leq m \leq N$ the function $\Gamma_{*} := \Gamma_{\sqrt{ \frac{2}{5}\log \left(\frac{N}{m} \log(\frac{N}{m}+1) \right)}}$ (the dependence on $m$ is implicit in this notation). There exists a subset $Q \subset \mathcal{O}(N)$ with $\mathbb{P}[Q] \geq 1 - C(e^{-c N^{2/5}})$ such that for all  $O = \{o_{i,n}\}_{1\leq i,n \leq N} \in Q$ the corresponding base change of $\{\phi_n \}_{n=1}^{N}$, that is
\[ \psi_n(x) := \sum_{i=1}^{N} o_{i,n} \phi_i(x), \]
satisfies the following.
For each $m$ in the range $1 \leq m \leq N$,
\[ \left|\left| \sum_{n=1}^{N} a_n \psi_n \right|\right|_{\Gamma_{*}} \ll \left(\sum_{n=1}^{N} a_n^2 \right)^{1/2}\]
for all vectors $\mathbf{a} \in \mathbb{R}^N$ such that $\textbf{support}(\mathbf{a}) \leq m$. (We use $\textbf{support}(\mathbf{a})$ to denote the number of nonzero coordinates of $\mathbf{a}$.)
\end{proposition}

The proof will build on arguments from \cite{Bour}, although the estimates we obtain are substantially stronger. We start by establishing a weaker result.
For a fixed $m$ in the range $1 \leq m \leq N$, we let $\mathbb{S}_m \subset \mathbb{R}^N$ denote the subset of vectors $\mathbf{b}$ such that $|| \mathbf{b} ||_{2} \leq 1$ and support$(\mathbf{b}) \leq m$.
We then define
\[ B(m, \mathcal{O}) :=  \sup_{\mathbf{a} \in \mathbb{S}_m}  || \sum_{n=1}^{N} a_n \psi_n ||_{\Gamma_{*}}.\]
Note that both the set $\mathbb{S}_m$ and the function $\Gamma_{*}:=\Gamma_{ \sqrt{\frac{2}{5}\log \left(\frac{N}{m} \log(\frac{N}{m}+1) \right)}}$ depend on $m$. Our first step will be to establish the following:

\begin{proposition}\label{prop:SingleM}For any $1 \leq m \leq N$ we have that
\[\mathbb{E}_{\mathcal{O}(N)} B(m, O) \ll 1\]
where the implied constant is independent of $m$ and $N$.
\end{proposition}

This does not quite give Proposition \ref{prop:main}, since there the claim is made with large probability and we require the estimates to hold for all $m$ simultaneously. The stronger claim, however, will be deduced later from the weaker statement using the concentration of measure phenomenon on the orthogonal group.

We will need the following result. This is Lemma 5.5 from \cite{Bour}. There it is attributed to \cite{BG}. The result is a concatenation of Lemma 1.10 and 1.12 in \cite{BG}. These are due to \cite{Chevet} and \cite{MP}, respectively.

\begin{lemma}\label{lem:OrthDecomp}Let $X$ and $Y$ be Banach spaces and consider the operator
\[T_{O} := \sum_{i,j =1}^{N} o_{ij}( x_i^{*} \otimes y_j )  \]
for $O := (o_{ij})_{1\leq i,j \leq N} \in \mathcal{O}(N)$, and where $\{x_i^{*}\}_{i=1}^{N}$ (respectively $\{y_j\}_{j=1}^{N}$) are sequences in $X^{*}$ (respectively $Y$). Then,
\begin{equation}\label{eq:operator}
\int_{\O(N)} ||T_{O} ||  \leq \frac{C \alpha(\{x_i^*\}_{i=1}^{N})}{\sqrt{N}} \int \left|\left| \sum_{j=1}^{N} g_j(\omega) y_j \right|\right| d\omega
+ \frac{C \alpha(\{y_j\}_{j=1}^{N})}{\sqrt{N}} \int \left|\left| \sum_{i=1}^{N} g_i(\omega) x_{i}^{*} \right|\right| d \omega
\end{equation}
where
\[ \alpha(\{x_i^{*}\} ) := \sup \{ ( \sum | \left< x_i^{*}, x \right>|^2 )^{1/2} : x \in X , ||x|| \leq 1 \}, \]
\[ \alpha(\{y_j\} ) := \sup \{ ( \sum | \left< y_j, y^* \right>|^2 )^{1/2} : y^{*} \in Y^* , ||y^{*}|| \leq 1 \}, \]
and $\{g_i\}_{i=1}^N$ is a system of independent Gaussians with mean zero and variance one. Note that the norms in (\ref{eq:operator}) refer respectively to the Banach spaces $B(X,Y)$, $Y$, and $X^*$.
\end{lemma}

Let $\ell^{2}[N]$ denote the set of real sequences $\mathbf{a} := \{a_n\}_{n=1}^{N}$. We will denote by $X$ the Banach space obtained by considering this set with the norm $|| \cdot ||_{[m]}$ defined as follows. For a vector $\mathbf{a}$, we define $||\mathbf{a}||_{[m]}$ to be the infimum of positive $c \in \mathbb{R}$ such that scaling the convex hull of $\mathbb{S}_m$ by $c$ results in a set containing $\mathbf{a}$.
We take $Y$ to be the space of real-valued functions on $\mathbb{T}$ equipped with the Orlicz norm associated to $\Gamma_{*}$.

Let $x_i^*$ ($1\leq i \leq N$) denote the canonical unit vectors in $\mathbb{R}^N$ (which is naturally identified with the dual space $X^*$). We have, from Lemma \ref{lem:OrthDecomp}, that
\[\mathbb{E} B(m,\mathcal{O}) \ll \frac{ \alpha(\{x^*_i\}_{i=1}^{N}) }{\sqrt{N}} \mathbb{E}||\sum g_i \phi_i ||_{\Gamma_*} + \frac{\alpha(\{\phi_i\}_{i=1}^{N})}{\sqrt{N}} \mathbb{E} || \sum g_i x_i^*||_{X^*}.\]

In order to establish Proposition \ref{prop:SingleM}, we need to show the above is $\ll 1$. This follows from the following estimates:
\[\alpha(\{x^*_i\}_{i=1}^{N}) \ll 1,\]
\[\alpha(\{\phi_i\}_{i=1}^{N}) \ll \left( \frac{N}{m} \log \left(\frac{N}{m} +1\right) \right)^{1/5}, \]
\[\mathbb{E}||\sum g_i \phi_i ||_{\Gamma_*}  \leq \sqrt{N}, \]
\[\mathbb{E} || \sum g_i x_i^*||_{X^*} \leq \sqrt{m} \sqrt{\log \left(\frac{N}{m}+1 \right)}. \]
The first estimate above follows from the observation that the convex hull of $\mathbb{S}_m$ is contained in the $\ell^2$ unit ball in $\mathbb{R}^N$.
We will prove the others in the following lemmas.

\begin{lemma}We have that $\mathbb{E}||\sum g_i \phi_i ||_{\Gamma_*}  \ll \sqrt{N} $.
\end{lemma}
\begin{proof}
Letting $C$ be a positive constant, by Fubini's theorem we have that $\mathbb{E} \int e^{ (\sum g_i \phi_i(x) )^2/(CN)} dx = \int \mathbb{E}  e^{ (\sum g_i \phi_i(x) )^2/CN} dx.$
Now, for each fixed $x$, we recall that $\sum_i |\phi_i(x)|^2 \leq N$, so $\frac{1}{\sqrt{CN}} \sum g_i \phi_i(x)$ is a Gaussian random variable with mean 0 and variance at most $\frac{1}{C}$. Thus, $\int \mathbb{E} e^{ (\sum g_i \phi_i(x) )^2/(CN)} dx \ll 1$ for an appropriate choice of $C$.

Since $e^{f^2/\lambda} \leq 1 + \frac{e^{f^2}}{\lambda}$ for $\lambda \geq 1$, we have that $\inf_{\lambda \in \mathbb{R}^{+}} \left\{ \int e^{|f  /\lambda|^2} \leq 2 \right\} \ll 1 + \int e^{|f |^2}$.
Applying this to $f = \frac{1}{\sqrt{CN}} \sum g_i \phi_i$, we have
\[ \left|\left|\frac{1}{\sqrt{CN}} \sum g_i \phi_i \right|\right|_{\Gamma_*} \leq \int e^{ (\sum g_i \phi_i(x) )^2/(CN)} dx.\]
Taking expectations on both sides, we have $\mathbb{E}||\sum g_i \phi_i ||_{\Gamma_*}  \ll \sqrt{N}$, as required.

\end{proof}

\begin{lemma}We have that $\alpha(\{\phi_i\}_{i=1}^{n}) \ll \left( \frac{N}{m} \log \left(\frac{N}{m}+1 \right) \right)^{1/5} $.
\end{lemma}
\begin{proof}

From Lemma \ref{lem:Gamma1} it follows that $||f||_{\Gamma_{*}} \leq \left( \frac{N}{m} \log \left(\frac{N}{m}+1 \right) \right)^{1/5} ||f||_{L^2}$. Now
\[||g||_{\Gamma_{*}^{*}} = \sup_{f \in \Gamma_{*}} \frac{\left<f,g\right>}{||f||_{\Gamma_{*} }} \geq \frac{\left<g,g\right>}{||g||_{\Gamma_{*} }}  \gg \frac{||g||^2_2 }{\left( \frac{N}{m} \log \left(\frac{N}{m}+1 \right) \right)^{1/5} ||g||_{2}}\]
\[ \gg \left( \frac{N}{m} \log \left(\frac{N}{m}+1 \right) \right)^{-1/5} ||g||_{2}.\]
Here we have used that the each element of the dual space $\Gamma_{*}^{*}$ can be represented as by integration against a measurable function.  This follows from standard properties of Orlicz spaces. In particular, see Theorem 14.2 of \cite{KR} since the modulus $\Gamma_{*}$ satisfies the $\Delta_{2}$ condition.

It now follows that if $ ||g||_{\Gamma_{*}^*} \leq 1 $ then $||g||_{2} \ll \left( \frac{N}{m} \log \left(\frac{N}{m}+1 \right) \right)^{1/5}$. Thus by Bessel's inequality we have
\[ \alpha(\{\phi_j\} ) := \sup \{ ( \sum | \left< \phi_i, g \right>|^2 )^{1/2} : g \in {\Gamma_{*}^{*}} , ||g||_{{\Gamma_{*}^{*}}} \leq 1 \} \ll\left( \frac{N}{m} \log \left(\frac{N}{m}+1 \right) \right)^{1/5}, \]
which completes the proof.
\end{proof}

\begin{lemma}\label{lem:1} We have that $\mathbb{E} || \sum g_i x_i^*||_{X^*} \leq \sqrt{m} \sqrt{\log \left(\frac{N}{m} +1\right)}$.
\end{lemma}
\begin{proof}
It follows from the definition of $X^*$ that
\[\mathbb{E} \left|\left| \sum g_i x_i^*\right|\right|_{X^*} = \mathbb{E} \sup_{\mathbf{a} \in \mathbb{S}_m } \left| \sum g_i a_i \right|.   \]
(Note that taking the supremum over the convex hull of $\mathbb{S}_m$ would yield the same result.)

The latter quantity is well studied in the theory of Gaussian processes. Recall that Dudley's bound \cite{D} gives
\[ \ll \int_{0}^{\infty} \sqrt{\log \left(\mathcal{N}(\mathbb{S}_m, \epsilon) \right)} d\epsilon,  \]
where $\mathcal{N}(\mathbb{S}_m, \epsilon)$ denotes the number of $\ell^2$ balls of radius $\epsilon$ needed to cover $\mathbb{S}_m$. Now clearly $\mathbb{S}_{m}$ is a subset of the $n$-dimensional $\ell^2$ unit ball, thus $\log \left(\mathcal{N}(\mathbb{S}_m, \epsilon) \right)=0$ for $\epsilon \geq 1$, and the above quantity is equal to
\[\int_{0}^{1} \sqrt{\log \left(\mathcal{N}(\mathbb{S}_m, \epsilon) \right)} d\epsilon.\]

Lemma \ref{lem:1} now follows from the following:

\begin{lemma}For $0<\epsilon \leq 1$, we have that
\[\mathcal{N}(\mathbb{S}_m, \epsilon)  \ll  {N\choose m} \left(\frac{3}{\epsilon} \right)^m  \]
and thus
\[\log \mathcal{N}(\mathbb{S}_m, \epsilon) \ll m \log \left( \frac{N}{m}+1 \right) + m \log \left(\frac{3}{\epsilon}\right).\]
\end{lemma}

\begin{proof}We only prove the first inequality (the second follows by taking logarithms). We let $K$ denote the unit $\ell^2$ ball in $\mathbb{R}^m$. Then $\mathcal{N}(K,\epsilon K) \leq \left(\frac{3}{\epsilon} \right)^{m}$, where $\mathcal{N}(K,\epsilon K)$ denotes the number of translates of $\epsilon K$ needed to cover $K$.
To see this, consider a maximal set of disjoint balls of radius $\frac{\epsilon}{2}$ with centers in $K$. Let $T$ denote the set of their centers. By maximality, taking balls of radius $\epsilon$ around each point in $T$ yields a cover of $K$, and hence the cardinality of $T$ is an upper bound on $\mathcal{N}(K,\epsilon K)$. Now, the union of all the disjoint balls of radius $\frac{\epsilon}{2}$ with centers in $T$ is a set with volume equal to $ |T| vol(\frac{\epsilon}{2}K)$, where $|T|$ denotes the cardinality of $T$ and $vol(\frac{\epsilon}{2}K)$ denotes the volume of the ball of radius $\frac{\epsilon}{2}$. Since this set is contained in $(1+ \frac{\epsilon}{2})K$, we have
\[\mathcal{N}(K,\epsilon K) \leq \frac{vol((1+\frac{\epsilon}{2})K)}{vol(\frac{\epsilon}{2}K)} = \frac{(1+\frac{\epsilon}{2})^m}{\left(\frac{\epsilon}{2}\right)^m} = \left(1+\frac{2}{\epsilon}\right)^m \leq \left(\frac{3}{\epsilon}\right)^m\]
whenever $ 0 < \epsilon \leq 1$.

Fix $m$ coordinates and consider the associated $m$-dimensional $\ell^2$ ball. We have shown that this can be covered by $\left(\frac{3}{\epsilon} \right)^{m}$ balls of radius $\epsilon$. Summing over all ${N\choose m}$ such balls completes the proof.
\end{proof}

This completes the proof of Lemma \ref{lem:1} and hence the proof of Proposition \ref{prop:SingleM}.
\end{proof}

\subsection{Concentration of Measure on $\mathcal{O}(n)$}

In the prior section, we proved that for any $1 \leq m \leq N$ we have $\mathbb{E}_{\mathcal{O}(N)} B(m,O) \ll 1$.
It follows from Markov's inequality that for some large universal $C$, we have $\mu(\mathcal{A}(m)) > \frac{1}{2}$, where
\[ \mathcal{A}(m) := \{ O \in \mathcal{O}(N) : B(m,O) \leq C \}\]
and $\mu(\mathcal{A}(m))$ denotes the measure of the set $\mathcal{A}(m)$ in $\mathcal{O}(N)$.

Consider the Hilbert-Schmidt norm on the set of $N\times N$ matrices, $||A||_{\text{HS}} := \left( \sum_{1 \leq i, j \leq N} |A_{i,j}|^2 \right)^{1/2}$. We recall the concentration of measure inequality on the Orthogonal group (see \cite{M}):

\begin{lemma}\label{lem:Concentration} Let $\mu$ denote the Haar measure on the orthogonal group $O(N)$ and $A \subset O(N)$ such that $\mu(A) > \frac{1}{2}$. Then,
\[\mathbb{P} \left[A \in O(N) : \inf_{B \in A_{c} } ||A-B||_{\text{HS}} > \epsilon \right] \ll e^{-c \epsilon^2 N}\]
for some absolute positive constant $c$.
\end{lemma}

For any $N \times N $ matrix $M = \{m_{i,j}\}$, using the bounds from Lemma \ref{lem:Gamma1} we have
\begin{eqnarray}\label{eq:HSnorm}
\nonumber
\left|\left| \sum_{1 \leq i, n \leq N} m_{i,n} a_i \phi_n \right|\right|_{\Gamma_{*}} & \ll &  \left(\frac{N}{m} \log\left( \frac{N}{m} \right) \right)^{1/5}   \left(  \sum_{n} \left( \sum_{i} m_{i,n} a_i \right)^{2} \right)^{1/2}\\
& \ll & \left(\frac{N}{m} \log\left( \frac{N}{m} \right) \right)^{1/5}  ||M||_{HS} ||\mathbf{a}||_{\ell^2}.
\end{eqnarray}
for all $\mathbf{a} \in \mathbb{R}^N$. The final inequality follows from Cauchy-Schwartz.

Now consider $\mathcal{A}(m,\epsilon) \subset \mathcal{O}(N)$, defined to be the set of all orthogonal matrices that differ from an element of $\mathcal{A}(m)$ by a matrix with Hilbert-Schmidt norm at most $\epsilon$. Using (\ref{eq:HSnorm}), we have that for $O \in \mathcal{A}\left(m,\left( \frac{m}{N \log( \frac{N}{m}) } \right)^{1/5} \right)$ we have $B(m,O) \leq C'$, where $C'$ is a new absolute constant. On the other hand, denoting the complement of $\mathcal{A}\left(m,\left( \frac{m}{N \log( \frac{N}{m}) } \right)^{1/5}  \right)$ by $\mathcal{A}^{c}\left(m,\left( \frac{m}{N \log( \frac{N}{m}) } \right)^{1/5} \right)$, by Lemma \ref{lem:Concentration} we have
\[\mathbb{P} \left[ O \in \mathcal{A}^{c}\left(m,\left( \frac{m}{N \log( \frac{N}{m}) } \right)^{1/5}  \right) \right] \ll e^{ -cN^{2/5}} \]
for some positive constant $c$.

Now to conclude the proof of Proposition \ref{prop:main}, it suffices to find a sufficiently high probability set of elements $O \in \mathcal{O}(N)$ such that for every $1\leq m \leq N$ we have $O \in \mathcal{A}\left(m,\left( \frac{m}{N \log( \frac{N}{m}) } \right)^{1/5}\right)$. However, for sufficiently large $N$, we see from the union bound that \[\mu\left( \bigcup_{1\leq m \leq N} \mathcal{A}^{c}\left(m,\left( \frac{m}{N \log( \frac{N}{m}) } \right)^{1/5}  \right) \right) \leq N e^{ -cN^{2/5}} \ll e^{ -c_{2} N^{2/5}}.\] This completes the proof of Proposition \ref{prop:main}.

\section{Maximal Function Decomposition}

\begin{proposition}\label{prop:decomp} For $N$  fixed, let $\{\phi_n(x)\}_{n=1}^N$ be an ONS such that $\sum_{n=1}^{N}|\phi_n(x)|^2 \leq N$. There exists $Q \subset \mathcal{O}(N)$ with $\mathbb{P}[Q] \geq 1 - C(e^{-cN^{2/5}})$ such that for $O\in Q$ the associated system $\Psi(O)=\{ \psi_{n}\}_{n=1}^{N}$ satisfies the following property. For any $f=\sum a_n \psi_n$, letting $m$ denote support$(\{a_n\})$, we have that the maximal function defined by
\[\mathcal{M}f :=  \sup_{I \subseteq [N]} \left| \sum_{n\in I} a_n \psi_n \right|  \]
can be decomposed as $\mathcal{M}f := \tilde{G} + \tilde{E}$ where $|| \tilde{G} ||_{\mathcal{G}} \ll ||f||_{2}$ and $||\tilde{E}||_{2} \ll \left(\frac{m}{N}\right)^{c}||f||_{2} $ for some universal constant $c>0$.
\end{proposition}

To prove this, we fix $Q \subset \mathcal{O}(N)$ from Proposition \ref{prop:main}. We now decompose $[N]$ into a family of subintervals according to a concept of mass defined with respect to the $a_i$ values.
We  define the \emph{mass} of a subinterval $I \subseteq [N]$ as $M(I) := \sum_{n \in I} |a_{n}|^2$. By normalization, we may assume that $M([N])=1$. We define $I_{0,1} := [N]$ and we iteratively define $I_{k,s}$, for $1\leq s\leq 2^k$, as follows. Assuming we have already defined $I_{k-1,s}$ for all $1 \leq s \leq 2^{k-1}$, we will define $I_{k,2s-1}$ and $I_{k,2s}$, which are subintervals of $I_{k-1,s}$. $I_{k,2s-1}$ begins at the left endpoint of $I_{k-1,s}$ and extends to the right as far as possible while covering strictly less than half the mass of $I_{k-1,s}$, while $I_{k,2s}$ ends at the right endpoint of $I_{k-1,s}$ and extends to the left as far as possible while covering at most half the mass of $I_{k-1,s}$. More formally,
we define $I_{k,2s-1}$ as the maximal subinterval of $I_{k-1,s}$ which contains the left endpoint of $I_{k-1,s}$ and satisfies $M(I_{k,2s-1}) < \frac{1}{2} M(I_{k-1,s})$. We also define $I_{k,2s}$ as the maximal subinterval of $I_{k-1,s}$ which contains the right endpoint of $I_{k-1,s}$ and satisfies $M(I_{k,2s}) \leq \frac{1}{2} M(I_{k-1,s})$. We note that these subintervals are disjoint. We may express $I_{k-1,s} = I_{k,2s-1} \bigcup I_{k,2s} \bigcup i_{k,s}$,  where $i_{k,s} \in I_{k-1,s}$. In other words, $i_{k,s}$ denotes the single element which lies between $I_{k,2s-1}$ and $I_{k,2s}$ (note that such a point always exists because we have required that $I_{k,2s-1}$ contains strictly less than half of the mass of the interval). Here it is acceptable, and in many instances necessary, for some choices of the intervals in this decomposition to be empty.
By construction we have that

\begin{equation}\label{eq:mass}
M(I_{k,s}) \leq 2^{-k}.
\end{equation}

We call an interval $J \subseteq [N]$ admissible if it is an element of the decomposition given above. We denote the collection of admissible intervals by $\mathcal{A}$. We additionally refer to the subset $\{I_{k,s}| 1\leq s\leq 2^k\}$ of $\mathcal{A}$ as the admissible intervals on level $k$ and the subset $\{i_{k,s} | 1 \leq s \leq 2^{k}\}$ as the admissible points on level $k$. We note that every point in $[N]$ is an admissible point on some level. (Eventually, we have subdivided all intervals down to being single elements.)

Now we write $\mathcal{I}_{k} :=\{I_{k,s} : 1\leq s\leq 2^k \}$. We decompose this as $\mathcal{I}_{k}^{a} :=\{I \in \mathcal{I}_{k} : |I| \leq 2^{-k/2}N \}$ and its complement, $\mathcal{I}_{k}^{b} :=\{I \in \mathcal{I}_{k} : |I| > 2^{-k/2}N \}$. Here, $|I|$ denotes the number of nonzero $a_i$ values contained in an interval $I$.

For $J \subseteq [N]$, we define
\[S_J(x) = \sum_{n\in J} a_n \psi_n(x).\]
We also define
\[ \tilde{S}_J(x) := \max_{I \subseteq J} \left|\sum_{n \in I} a_n \psi_n(x)\right| .\]

From Lemma \ref{lem:Gamma2} and Proposition \ref{prop:main}, we deduce that $S_J = G_J + E_J $ where $||G_J||_{\mathcal{G}} \ll ||S_J||_{2} $ and $||E_J||_{2} \ll \left(\frac{|J|}{N}\right)^{c'} ||S_J||_{2}  $ for some positive constant $c'$. Our purpose now is to show a similar decomposition for $\tilde{S}_J(x)$. Clearly, it suffices to show such a decomposition for a pointwise majorant.  Denote the decomposition of $S_{I_{k,s}}$ by  $S_{I_{k,s}}:= G_{k,s}+E_{k,s}$, and the decomposition of $S_{i_{k,s}}$ by $S_{i_{k,s}} := G_{i_{k,s}} + E_{i_{k,s}}$. Setting $r =3$, for an interval $J$ we have the following bound, where the sums below are restricted to values of $k,s$ such that $I_{k,s}, i_{k,s} \subseteq J$:
\[\tilde{S}_J (x)  \ll  \sum_{k} \left( \sum_{s} |G_{k,s} +E_{k,s}|^r\right)^{1/r} + \sum_k \left( \sum_{s} | G_{i_{k,s}} + E_{i_{k,s}} |^r   \right)^{1/r} \]
\begin{eqnarray}\label{eqn:decomp}\nonumber
  & \ll &  \left(\sum_{k} \left( \sum_{s} |G_{k,s} |^r   \right)^{1/r} +  \sum_{k} \left( \sum_{s} |G_{i_{k,s}} |^r   \right)^{1/r}\right) \\
  &+&  \left( \sum_{k} \left( \sum_{s} | E_{k,s} |^r   \right)^{1/r} + \sum_{k} \left( \sum_{s} | E_{i_{k,s}} |^r   \right)^{1/r} \right) =: \tilde{G}_{J} + \tilde{E}_J.
\end{eqnarray}
This follows from the observation that for each point $x$, the maximizing subinterval $I \subseteq J$ can be decomposed as a union of admissible intervals and points with at most two intervals and points on each level. The contribution on each level can then be bounded by a constant times the contribution from the ``worst" interval/point, which is in turn bounded by the quantity inside the sum over $k$ above for each level $k$.

For an admissible interval $J$, we let $k^*$ denote the level of $J$. We note that the sums over $k$ in (\ref{eqn:decomp}) range only over $k \geq k^*$ (and the sums over $s$ are also appropriately restricted).
Next we show that $||\tilde{G}_{J}||_{\mathcal{G}(c)} \ll ||S_J||_{2}$ for some absolute constant $c$ and $||\tilde{E}_J ||_{2} \ll \left(\frac{|J|}{N}\right)^{c'} ||S_J||_{2}$.

Now let us estimate $||\tilde{E}_{J}||_{2}$. We first estimate the contribution from the admissible points $i_{k,s} \in J$. We observe
\[\left|\left| \sum_k \left(\sum_s |E_{i_{k,s}}|^r\right)^{\frac{1}{r}} \right|\right|_2 \leq \sum_k \left|\left| \left(\sum_s |E_{i_{k,s}}|^r\right)^{\frac{1}{r}} \right|\right|_2.\]
Since $r>2$, this is
\[\leq \sum_k \left(\sum_s ||E_{i_{k,s}}||_2^2\right)^{\frac{1}{2}} \ll \left(\frac{1}{N}\right)^{c'} \sum_k \left(\sum_s ||S_{i_{k,s}}||^2_2\right)^{\frac{1}{2}},\]
where the latter inequality follows from the definition of $E_{i_{k,s}}$.

Now since these sums only range over values of $k,s$ such that $i_{k,s} \in J$, we may split the sum over $k$ into two portions as:
\begin{equation}\label{eq:points1}
\sum_k \left(\sum_s ||S_{i_{k,s}}||^2_2\right)^{\frac{1}{2}} = \sum_{k = k^*}^{k^*+ 10 \log(N)} \left(\sum_s ||S_{i_{k,s}}||^2_2\right)^{\frac{1}{2}} + \sum_{k > k^* + 10 \log(N)} \left(\sum_s ||S_{i_{k,s}}||^2_2\right)^{\frac{1}{2}}.
\end{equation}
To bound the first quantity in (\ref{eq:points1}), it suffices to observe that the inner quantity for each $k$ is at most $||S_J||_2$, and hence its contribution is $\ll \log(N) ||S_J||_2 \ll N^{\epsilon} ||S_J||_2$, for a constant $\epsilon < c'$. (Thus we will adjust the value of $c'$ for our final estimate by subtracting $\epsilon$.)

To bound the second quantity in (\ref{eq:points1}), we note that for any $i_{k,s} \in J$ with $k > k^*+10\log(N)$, we have $||S_{i_{k,s}}||_2^2 \leq N^{-10} ||S_J||^2_2$. There are at most $N$ points $i_{k,s}$ in the sum, and thus
\[\sum_{k > k^* + 10 \log(N)} \left(\sum_s ||S_{i_{k,s}}||^2_2\right)^{\frac{1}{2}} \ll N^{-4} ||S_{J}||_2.\]

To estimate the contribution from the admissible intervals, we proceed as follows.
For each $k \geq k^*$, we define $I_{k}^{a}(J)$ to be the set of admissible intervals $I$ on level $k$ contained in $J$ such that $|I| < 2^{-(k-k^*)/2}|J|$ and we let $I_{k}^{b}(J)$ denote the set of remaining admissible intervals on level $k$ contained in $J$. Note that $I_k^a(J)$ and $I_k^b(J)$ are disjoint, and their union is the set of all admissible intervals on level $k$ contained in $J$. It thus suffices to estimate
\[ \tilde{E}_{J}^{a}+\tilde{E}_{J}^b := \sum_{k\geq k^*} \left( \sum_{I_{k,s} \in I_{k}^{a}(J)} | E_{k,s} |^r   \right)^{1/r}  + \sum_{k} \left( \sum_{I_{k,s} \in I_{k}^{b}(J)} | E_{k,s} |^r   \right)^{1/r}.  \]

Now $|I_{k}^{b}(J)| \leq 2^{(k-k^*)/2}$, and we also have
\[||E_{k,s}||_{2} \ll \left(\frac{|J|}{N}\right)^{c'} ||S_{k,s}||_2 \ll \left(\frac{|J|}{N}\right)^{c'} 2^{-(k-k^*)/2} ||S_J||_2.\]
Since $r>2$, we have:
\[\left|\left| \sum_{k\geq k^*} \left( \sum_{s\in I_{k}^{b}(J)} | E_{k,s} |^r   \right)^{1/r}  \right|\right|_{2}  \leq \sum_{k\geq k^*} \left( \sum_{s\in I_{k}^{b}(J)} ||E_{k,s}||_{2}^2\right)^{1/2} \]
\[ \ll \left(\frac{|J|}{N}\right)^{c'} ||S_{J}||_{2} \sum_{j\geq 0} 2^{-j/4} \ll \left(\frac{|J|}{N}\right)^{c'} ||S_{J}||_{2}. \]

Next, we recall that $I \in I_{k}^{a}(J)$ implies $|I|\leq 2^{-(k-k^*)/2}|J|$. We have  $||S_{I_{k,s}}||_{2} \ll 2^{-(k-k^*)/2} ||S_{J}||_{2} $, thus
$||E_{k,s}||_{2} \ll \left(\frac{|J|}{N}\right)^{c'} 2^{-c'(k-k^*)/2} ||S_{I_{k,s}}||_2 \ll \left(\frac{|J|}{N}\right)^{c'} 2^{-(c'+1)(k-k^*)/2} ||S_J||_2$.

We then have
\[\left|\left| \sum_{k\geq k^*} \left( \sum_{I_{k,s} \in I_{k}^{a}(J)} | E_{k,s} |^r   \right)^{1/r}  \right|\right|_{2}  \leq \sum_{k\geq k^*} \left( \sum_{I_{k,s}\in I_{k}^{a}(J)} ||E_{k,s}||_{2}^2\right)^{1/2} \]
\[ \ll \left(\frac{|J|}{N}\right)^{c'} ||S_{J}||_{2} \sum_{k\geq k^*} 2^{k-k^*} 2^{-(c'+1)(k-k^*)} \ll \left(\frac{|J|}{N}\right)^{c'} ||S_{J}||_{2}.\]
Here we have used the fact that there are at most $2^{k-k^*}$ values of $s$ such that $I_{k,s} \subseteq J$ for each $k\geq k^*$.
We can apply this for $J = [N]$ in particular, recalling that $|J|$ denotes the number of nonzero $a_i$ values contained in $J$, which in this case is $m$.
This completes the proof that $||\tilde{E}||_{2} \ll \left(\frac{m}{N}\right)^{c'}||f||_{2}$ for some positive constant $c'$.

To show that $|| \tilde{G} ||_{\mathcal{G}(c)} \ll ||f||_{2}$ for some universal constant $c>0$, we will use the following lemma. These implications and arguments are well-known, however we include a proof for completeness.

\begin{lemma}\label{lem:tricky} Let $A$ denote a fixed, positive constant. For positive constants $c,C$, we define the following sets of measurable functions:
\[S_1(c) := \{ f:\mathbb{T}\rightarrow \mathbb{R} \text{ s.t. } || f||_{p} \leq c \sqrt{p}A \; \; \forall p \geq 2\},\]
\[S_2(c,C) := \{ f:\mathbb{T} \rightarrow \mathbb{R} \text{ s.t. } \mu(|f| \geq \lambda) \leq Ce^{-c\frac{\lambda^2}{A^2}}\; \; \forall \lambda \geq 0\},\]
\[S_3(c):= \{f:\mathbb{T}\rightarrow \mathbb{R} \text{ s.t. } ||f||_{\mathcal{G}(c)} \leq A\},\]
where $\mu(|f| \geq \lambda)$ denotes the Lebesgue measure of the subset of $x \in \mathbb{T}$ such that $|f(x)| \geq \lambda$.
Then for any $c>0$, there exist positive constants $c', C', c''$ (depending only on $c$) such that $S_1(c) \subseteq S_2(c',C')$ and $S_1(c) \subseteq S_3(c'')$. Similarly, for any $c, C>0$, there exist positive constants $c', c''$ (depending only on $c, C$) such that $S_2(c,C) \subseteq S_1(c')$ and $S_2(c,C) \subseteq S_3(c'')$. Finally, for any $c>0$, there exist positive constants $c', C', c''$ (depending only on $c$) such that $S_3(c) \subseteq S_2(c',C')$ and $S_3(c) \subseteq S_1(c'')$.
\end{lemma}

\begin{proof} Fixing $c,C$, we will determine $c'$ such that $S_2(c,C) \subseteq S_3(c')$ (for every $A$). We consider an $f \in S_2(c,C)$.
We consider $c' := d_1 d_2$ as a product of two variables $d_1, d_2$ whose values will be set later. We assume $d_1 \leq 1$. We have:
\begin{equation}\label{eq:step1}
 \int_{\mathbb{T}} e^{c'|f|^2/A^2} = \int_{\mathbb{T}}e^{d_1 d_2 |f|^2/A^2} \leq 1+ d_1 \int_{\mathbb{T}} e^{d_2 |f|^2/A^2},
\end{equation}
using the inequality $e^{x/a} \leq \frac{1}{a} e^x+1$ for all $a\geq 1$ and non-negative $x$ (this can be seen by considering the Taylor expansion of $e^x$).

Now, we observe that
\[\int_{\mathbb{T}} e^{d_2 |f|^2/A^2} \leq \sum_{k\geq 0} \int_{\mathbb{T}} e^{d_2 |f|^2/A^2}\cdot \mathbb{I}_{A^2 k\leq |f|^2 < A^2(k+1)} \leq \sum_{k\geq 0} \mu(|f|^2 \geq A^2 k) e^{d_2 (k+1)},\]
where $\mathbb{I}_{A^2 k\leq |f|^2 < A^2(k+1)}$ denotes the characteristic function of the set on which $|f|^2$ takes values between $A^2 k$ and $A^2 (k+1)$.
Since $f \in S_2(c,C)$, we have $\mu(|f|^2 \geq A^2 k) \leq Ce^{-ck}$ for all $k \geq 0$. Thus, we conclude
\[\int_{\mathbb{T}} e^{d_2 |f|^2/A^2} \leq \sum_{k\geq 0} Ce^{-ck + d_2(k+1)} = Ce^{d_2} \sum_{k\geq 0} e^{-(c-d_2)k} = \frac{Ce^c}{e^{c-d_2}-1}\]
whenever $d_2 < c$.
Setting $d_2 = c/2$, we obtain $\leq Ce^c/(e^{c/2}-1)$. Letting $d_1 = \min \left\{1,\frac{e^{c/2}-1}{Ce^c}\right\}$, we have
\[d_1 \int_{\mathbb{T}} e^{d_2 |f|^2/A^2} \leq 1,\]
and hence $\int_{\mathbb{T}} e^{c'|f|^2/A^2}-1 \leq 1$ for $c' = d_1 d_2$, showing that $f \in S_3(c')$. Note that $c' = d_1d_2$ depends only on $c$ and $C$.

Conversely, we observe that for every $c>0$, $S_3(c) \subseteq S_2(c,2)$. To see this, consider $f \in S_3(c)$. Then we have
\[ \int_{\mathbb{T}} e^{c |f|^2/A^2} -1 \leq 1 \Rightarrow \int_{\mathbb{T}} e^{c|f|^2/A^2} \leq 2.\]
Thus for any $\lambda > 0$,
\[ \mu(|f|\geq \lambda)e^{c \lambda^2/A^2} \leq \int_{\mathbb{T}} e^{c |f|^2/A^2} \leq 2.\]
It follows that $f \in S_2(c,2)$.

For any $c>0$, we will now show there exist $c',C$ such that $S_1(c) \subseteq S_2(c',C)$ (for every $A$). We consider an $f \in S_1(c)$. This means that $||f||_p ^p \leq c^p p^{\frac{p}{2}} A^p$ for all $p \geq 2$. Thus, for every $\lambda > 0$, $\mu(|f|\geq \lambda) \lambda^p \leq (cA)^p p^{\frac{p}{2}}$, which implies
\begin{equation}\label{eq:minimize}
\mu(|f|\geq \lambda) \leq \frac{(cA)^p p^{\frac{p}{2}}}{\lambda^p}.
\end{equation}
For a fixed $\lambda$, we may minimize this quantity over the choices of $p \geq 2$. In the case that $\frac{\lambda^2}{e c^2A^2} \geq 2$, we may set $p$ equal to this value, and the quantity in (\ref{eq:minimize}) then becomes:
\[ \left(\frac{cA}{\lambda}\right)^{\frac{\lambda^2}{ec^2A^2}} \left(\frac{\lambda^2}{ec^2A^2}\right)^{\frac{\lambda^2}{2ec^2A^2}} = e^{-\frac{\lambda^2}{2ec^2A^2}}.\]
Hence by setting $c' = \frac{1}{2ec^2}$, we achieve $\mu(|f|\geq \lambda) \leq e^{-c' \lambda^2/A^2}$ in these cases.

Now, when $\frac{\lambda^2}{e c^2A^2} < 2$, we note that $e^{-c' \lambda^2/A^2} \geq e^{-c' (2ec^2)} = e^{-1}$. Thus, setting $C =e$, we have
$\mu(|f|\geq \lambda) \leq 1 \leq C e^{-c' \lambda^2/A^2}$ in these cases.
Hence, in all cases we have that
\[ \mu(|f|\geq \lambda) \leq C e^{-c' \lambda^2/A^2},\]
so $f \in S_2(c',C)$.

Conversely, for any $c, C > 0$, we will show there exists $c'$ such that $S_2(c,C) \subseteq S_1(c')$ for every $A$. We consider an $f \in S_2(c,C)$. Then for every $\lambda \geq 0$, we have $\mu(|f| \geq \lambda) \leq Ce^{-c\frac{\lambda^2}{A^2}}$. We fix $p \geq 2$. We observe:
\[ ||f||_p^p = p \int_0^\infty \lambda^{p-1} \mu(|f| > \lambda) d\lambda \ll p \int_0^\infty \lambda^{p-1} e^{-c \lambda^2/A^2} d\lambda.\]
Substituting $\lambda = t^{\frac{1}{p}}$, we see this equals
\begin{equation}\label{eq:gamma}
\int_0^\infty e^{-c t^{\frac{2}{p}}/A^2} dt.
\end{equation}

We note that identity $\frac{p}{2} \Gamma\left(\frac{p}{2}\right) = \int_0^\infty e^{-s^{\frac{2}{p}}} ds$
where $\Gamma$ denotes the function $\Gamma(z) := \int_0^\infty y^{z-1}e^{-y} dy$. Setting $s = \left(\frac{c}{A^2}\right)^{\frac{p}{2}} t$, we see that the quantity in (\ref{eq:gamma}) is
\[ = c^{-\frac{p}{2}} A^p \int_0^\infty e^{-s^{\frac{2}{p}}} ds = c^{-\frac{p}{2}}A^p \left(\frac{p}{2}\right) \Gamma\left(\frac{p}{2}\right).\]
By Sterling's formula, $\Gamma \left(\frac{p}{2}\right) \ll p^{-1/2} \left(\frac{p}{2e}\right)^{\frac{p}{2}}$.
Hence,
\[||f||_p \ll A \sqrt{p} \left(p^{\frac{1}{2p}}\right) \ll A \sqrt{p},\] as required.

\end{proof}

%
%
%

Appealing to Lemma \ref{lem:tricky}, we see that we may bound the quantity $||\tilde{G}_{J}||_{\mathcal{G}(c)}$ by considering the $p$ norm. We recall that
\[\tilde{G}_{J} = \sum_{k} \left( \sum_s |G_{k,s}|^{r} \right)^{1/r} + \sum_{k} \left( \sum_s |G_{i_{k,s}}|^{r} \right)^{1/r} ,\]
where the sums are restricted to values of $k,s$ such that $I_{k,s},i_{k,s} \subseteq J$. We let $k^*$ again denote the level of $J$, so we are only summing over values $k \geq k^*$.

We have
\[  \left|\left| \sum_{k}  \left(\sum_{s} |G_{k,s}|^{r} \right)^{1/r} +  \sum_{k} \left( \sum_s |G_{i_{k,s}}|^{r} \right)^{1/r}\right|\right|_{p}  \leq \sum_k \left|\left| \left(\sum_{s} |G_{k,s}|^{r} \right)^{1/r}\right|\right|_{p} + \sum_k \left|\left| \left(\sum_{s} |G_{i_{k,s}}|^{r} \right)^{1/r}\right|\right|_{p} \]
by the triangle inequality, and this is
\[ = \sum_k \left|\left| \sum_s |G_{k,s}|^r\right|\right|^{\frac{1}{r}}_{\frac{p}{r}} + \sum_k \left|\left| \sum_s |G_{i_{k,s}}|^r\right|\right|^{\frac{1}{r}}_{\frac{p}{r}} \leq \sum_k \left(\sum_s \left|\left| |G_{k,s}|^r\right|\right|_{\frac{p}{r}} \right)^{\frac{1}{r}} + \sum_k \left(\sum_s \left|\left| |G_{i_{k,s}}|^r\right|\right|_{\frac{p}{r}} \right)^{\frac{1}{r}}\]
\[= \sum_k \left(\sum_s ||G_{k,s}||^r_p\right)^{\frac{1}{r}} + \sum_k \left(\sum_s ||G_{i_{k,s}}||^r_p\right)^{\frac{1}{r}}\]
by another application of the triangle inequality.

Now, using that $||G_{k,s}||_p \ll \sqrt{p} ||S_{I_{k,s}}||_2$ and $||G_{i_{k,s}}||_p \ll \sqrt{p} ||S_{i_{k,s}}||_2$ by Lemma \ref{lem:tricky} and $||S_{I_{k,s}}||_{2} \ll ||S_J||_2 2^{-(k-k^*)/2}$ and $||S_{i_{k,s}}||_2 \ll ||S_J||_2 2^{-(k-k^*)/2}$, we have
\[ ||\tilde{G}_{J}||_{p}  \leq \sum_{k\geq k^*} \left( \sum_{s} ||   G_{k,s} ||_{p}^{r} \right)^{1/r} + \sum_{k\geq k^*} \left( \sum_{s} ||   G_{i_{k,s}} ||_{p}^{r} \right)^{1/r}  \ll \sqrt{p} ||S_J||_2 \sum_{k\geq k^*} \left(\sum_s 2^{-r(k-k^*)/2}\right)^{\frac{1}{r}} .\]
Since the sum of $s$ ranges over at most $2^{k-k^*}$ values (recall we only include values of $s$ such that $I_{k,s} \subseteq J$) and $r > 2$, this is
\[ \ll \sqrt{p}||S_J||_2 \sum_{k\geq k^*} 2^{(k-k^*) (r^{-1}-2^{-1})} \ll \sqrt{p}||S_J||_2. \]
It thus follows from Lemma \ref{lem:tricky} that
\[ ||\tilde{G}_{J}||_{\mathcal{G}(c)} \ll ||S_{J}||_{2} \]
for some positive constant $c$. Lastly, we have that $||\tilde{G}_{J}|| \ll ||\tilde{G}_{J}||_{\mathcal{G}(c)}$ from the definition of the Orlicz norm.

\section{Proof of the Main result}

%
%
%
We are now ready to prove:

\begin{theorem}Let $\Phi := \{\phi_n(x)\}_{n=1}^{N}$ be an ONS such that $\sum_{n=1}^{N}|\phi_n(x)|^2 \leq N$. Then there exists $Q \subset \mathcal{O}(N)$ with $\mathbb{P}[Q] \geq 1 - Ce^{-cN^{2/5}}$ such that for $O\in Q$ the alternate ONS $\Phi(O)$ satisfies
\[ || \mathcal{V}^{2}f ||_{2} \ll \sqrt{\log \log (N)} ||f||_{2}.\]
\end{theorem}

Here we use the mass decomposition (into dyadic subintervals $I_{k,s}$) stated previously. We use the following easily verified fact (see \cite{LewkoJFA}, Lemma 29):

\begin{lemma}\label{lem:decomposition} For every $J \subseteq [N]$, ($J \neq \emptyset$) there exist $\tilde{J}_\ell, \tilde{J}_r \in \mathcal{A}$ and $i_J \in [N]$ such that $\tilde{J}:= \tilde{J}_{\ell} \cup i_J \cup \tilde{J}_r$ is an interval (i.e. $J_\ell, i_J, J_\ell$ are adjacent), $J \subseteq \tilde{J}$, and $M(\tilde{J}) \leq 2M(J)$.
\end{lemma}

Without loss of generality, we set $||f||_{2}=1$, and we have the pointwise inequality
\[|\mathcal{V}^{2}f(x)|^2  \ll  \sum_{k,s} |\tilde{S}_{I_{k,s}} \mathbb{I}_{B(I_{k,s})}|^2  + \sum_{k,s}| S_{i_{k,s}}|^2 + \log \log(N), \]
where $B(I_{k,s}) \subseteq \mathbb{T}$ is the set such that $|\tilde{S}_{I_{k,s}}(x)|^2 \geq C \log\log(N) M(I_{k,s})$, for a fixed constant $C$ whose value will be chosen to be sufficiently large.
Appealing to Proposition \ref{prop:decomp}, for each $I_{k,s}$ we can decompose $\tilde{S}_{I_{k,s}} = \tilde{G}_{I_{k,s}} + \tilde{E}_{I_{k,s}}$.
We then define $B_{G}(I_{k,s}) \subseteq \mathbb{T}$ by  $|\tilde{G}_{I_{k,s}}(x)|^2 \geq \frac{C}{10} \log\log(N) M(I_{k,s})$ and $B_{E}(I_{k,s}) \subseteq \mathbb{T}$ by  $|\tilde{E}_{I_{k,s}}(x)|^2 \geq \frac{C}{10} \log\log(N) M(I_{k,s})$.

Clearly $\int \sum_{k,s}| S_{i_{k,s}}|^2 \leq 1$ is acceptable, so it suffices to show that
\[ \int\sum_{k,s} |\tilde{S}_{I_{k,s}} \mathbb{I}_{B(I_{k,s})}|^2 \ll 1.\]
Now appealing to the decomposition above, we have
\[ \int\sum_{k,s} |\tilde{S}_{I_{k,s}} \mathbb{I}_{B(I_{k,s})}|^2 \ll  \int\sum_{k,s} | \tilde{G}_{I_{k,s}}  \mathbb{I}_{B_{G}(I_{k,s})}|^2  +  \int\sum_{k,s}  | \tilde{E}_{I_{k,s}}  \mathbb{I}_{B_{E}(I_{k,s})} |^2   .\]

First we estimate
\[\int\sum_{k,s}  | \tilde{E}_{I_{k,s}}  \mathbb{I}_{B_{E}(I_{k,s})} |^2 \ll \int\sum_{k,s}  | \tilde{E}_{I_{k,s}} |^2. \]
Employing notation previously used above, we let $I_{k}^{a} := \{I_{k,s} \text{ s.t. } |I_{k,s}| \leq 2^{-k/2}N \}$ and $I_{k}^{b} := \{I_{k,s} \text{ s.t. } |I_{k,s}| > 2^{-k/2}N \}$. Thus $I \in I_{k}^{a}$ implies $|I| \leq 2^{-k/2}N$ and $|I_{k}^{b}|\leq 2^{k/2}$. We then have
\[\int\sum_{k,s}  | \tilde{E}_{I_{k,s}} |^2 = \int\sum_{I_{k,s} \in I_{k}^{a}}  | \tilde{E}_{I_{k,s}} |^2 + \int\sum_{ I_{k,s} \in I_{k}^{b}}  | \tilde{E}_{I_{k,s}} |^2 . \]

Using that $I \in I_{k}^{a}$ implies $|I| \leq 2^{-k/2}N$, we have  $ \int |\tilde{E}_{I_{k,s}}|^2 \ll 2^{-c' k/2}||S_{I_{k,s}} ||_{2}^2 \ll 2^{-k - c'k/2} $. Thus
\[\int\sum_{I_{k,s} \in I_k^a}  | \tilde{E}_{I_{k,s}} |^2  \ll \sum_{k} 2^{-c'k/2} \ll 1. \]
Next, using that $|I_{k}^{b}|\leq 2^{k/2}$ and $\int | \tilde{E}_{I_{k,s}} |^2 \ll 2^{-k}$, we have
\[\int\sum_{ I_{k,s} \in I_{k}^{b}}  | \tilde{E}_{I_{k,s}} |^2  \ll \sum_{k} 2^{-k/2} \ll 1. \]

Finally, we estimate
\[ \int\sum_{k,s} | \tilde{G}_{I_{k,s}}  \mathbb{I}_{B_{G}(I_{k,s})}|^2.\]
We can choose $C$ sufficiently large so that $|B_{G}(I_{k,s})| \ll \frac{1}{\log^{10}(N)} $ for all $k,s$ (here, $|B_{G}(I_{k,s})|$ denotes the Lebesgue measure). To see this, recall that $||\tilde{G}_{I_{k,s}}||_{\mathcal{G}(c)} \ll \sqrt{M(I_{k,s})}$. By Lemma \ref{lem:tricky}, there exists a constant $c'>0$ such that
\[ \mu\left( |\tilde{G}_{I_{k,s}}| \geq \lambda\right) \ll e^{-c' \lambda^2 / M(I_{k,s})}\]
for all $\lambda \geq 0$. Setting $\lambda^2 = \frac{C}{10} \log \log(N) M(I_{k,s})$, we obtain
\[ |B_G(I_{k,s})| \ll \log(N)^{-c'C/10}.\]
We can then choose $C$ sufficiently large with respect to $c'$ make this estimate $\ll \frac{1}{\log^{10}(N)}$.

Now we split the sum at $k=100 \log(N)$ so
\[  \int\sum_{k,s} | \tilde{G}_{I_{k,s}}  \mathbb{I}_{B_{G}(I_{k,s})}|^2  =  \int\sum_{\substack{k,s \\ k \geq 100\log(N) }} | \tilde{G}_{I_{k,s}}  \mathbb{I}_{B_{G}(I_{k,s})}|^2 + \int\sum_{\substack{k,s \\ k < 100\log(N) }} | \tilde{G}_{I_{k,s}}  \mathbb{I}_{B_{G}(I_{k,s})}|^2   .   \]

By the Cauchy-Schwarz inequality,
\[ \int\sum_{\substack{k,s \\ k < 100\log(N) }} | \tilde{G}_{I_{k,s}} \mathbb{I}_{B_{G}(I_{k,s})}|^2 \ll \sum_{k,s} || \tilde{G}_{I_{k,s}}||_{4}^2\;  || 1_{B_{G}(I_{k,s})}||_{4}^2.   \]
Now, by Lemma \ref{lem:tricky}, we have $|| \tilde{G}_{I_{k,s}} ||_{4}^2 \ll || S_{I_{k,s}} ||_{2}^2 \ll 2^{-k} $ and, by the previous estimate, $|| \mathbb{I}_{B_{G}(I_{k,s})}||_{4}^2 \ll  \frac{1}{\log^{5}(N)}$. Thus we have shown that
the quantity above is
\[ \ll \frac{1}{\log^{5}(N)}  \int\sum_{\substack{k,s \\ k < 100\log(N) }} || \tilde{G}_{I_{k,s}}||_{4}^2 \ll \frac{1}{\log^{4}(N)} \ll 1. \]

Lastly, let $T \subset [N]$ denote the set of indices appearing in some $I_{k,s}$ for $k \geq 100\log(N)$. Note that any index will appear in at most $N$ such intervals, and that $M(I_{k,s}) \leq N^{-100}$ if $k \geq 100\log(N) $. Thus $|a_n|^2 \ll N^{-100}$ for $n \in T$. Thus we have
\[\int\sum_{\substack{k,s \\ k \geq 100\log(N) }} | \tilde{G}_{I_{k,s}}  \mathbb{I}_{B_{G}(I_{k,s})}|^2 \ll N^2 \int \sum_{n \in T} |a_n \phi_n(x)|^2 \ll N^{-98} \int \sum_{n \in T} | \phi_n(x)|^2 \ll 1 .\]
This completes the proof.

\texttt{A. Lewko, Department of Computer Science, The University of Texas at Austin}

\textit{alewko@cs.utexas.edu}
\vspace*{0.5cm}

\texttt{M. Lewko, Department of Mathematics, The University of Texas at Austin}

\textit{mlewko@math.utexas.edu}

\end{document}